\documentclass[12pt]{article}%
  
\usepackage{amsmath,enumerate}
\usepackage{amsfonts}
\usepackage{amssymb}

\newtheorem{theorem}{Theorem}[section]

\newtheorem{lemma}[theorem]{Lemma}

\newtheorem{proposition}[theorem]{Proposition}

\newenvironment{proof}[1][Proof]{\noindent\textbf{#1.} }
{\hfill \ \rule{0.5em}{0.5em}}

\begin{document}

\title{Upper bounds for $B_h[g]$-sets with small $h$}
\author{Craig Timmons\thanks{Department of Mathematics and Statistics, California State University Sacramento.  
This work was supported by a grant from the Simons Foundation (\#359419).} }

\date{}

\maketitle

\begin{abstract}
For $g \geq 2$ and $h \geq 3$, we give small improvements on the maximum size of a $B_h[g]$-set 
contained in the interval $\{1,2, \dots , N \}$.  In particular, we show that a $B_3[g]$-set 
in $\{1,2, \dots , N \}$ has at most $(14.3 g N)^{1/3}$ elements.  The previously best known 
bound was $(16 gN)^{1/3}$ proved by Cilleruelo, Ruzsa, and Trujillo.  We also introduce a related optimization problem 
that may be of independent interest.    
\end{abstract}


\section{Introduction}

Let $A \subseteq [N]:= \{1,2, \dots , N \}$ and let $h$ and $g$ be positive integers.  We say that $A$ is a \emph{$B_h[g]$-set} if 
for any integer $n$, there are at most $g$ distinct multi-sets $\{a_1 , a_2 , \dots , a_h \}  \subseteq A$ such that 
\[
a_1 + a_2 + \dots + a_h = n.
\]
Determining the maximum size of a $B_h[g]$-set in $A \subseteq [N]$ is a well-studied problem in number theory.  
Initial bounds on $B_h[g]$-sets were obtained combinatorially.  
Indeed, if $A$ is a $B_h[g]$-set, then consider the $\binom{ |A| + h - 1 }{h}$ multi-sets of size $h$ in $A$.  
The sum of the elements in each of the multi-sets represents each integer in $\{1,2, \dots , h N \}$ at most $g$ times.  Therefore,  
\begin{equation}\label{trivial bound}
\binom{ |A| + h  - 1}{ h } \leq  g h N 
\end{equation}
which implies $|A| \leq ( h!  g h N )^{1/h}$.
The breakthrough papers of 
Cilleruelo, Ruzsa, Trujillo \cite{crt}, Cilleruelo, Jim\'{e}nez-Urroz \cite{cu}, and  Green \cite{green} introduced methods from analysis and probability to obtain significant improvements on (\ref{trivial bound}).  Several of the results in these papers 
have yet to be improved upon.  For more on $B_h[g]$-sets, we recommend the survey papers of 
O'Bryant \cite{o} and Plagne \cite{p}.
We will be concerned with $B_h[g]$-sets where $g \geq 2$ and $h \geq 3$.    
For $3 \leq h \leq 6$ and $g \geq 2$, the best known upper bound on the size of a $B_h[g]$-set $A \subseteq [N]$ is  
\begin{equation}\label{bound 1}
|A| \leq \left( \frac{ h! h g N }{ 1 + \cos^h ( \pi / h ) } \right)^{1/h}
\end{equation}
due to Cilleruelo, Ruzsa, and Trujillo \cite{crt}.  For $h \geq 7$, the best known bound is 
\begin{equation}\label{bound 2}
|A| \leq \left( \sqrt{ 3h } h! g N \right)^{1/h}
\end{equation}
which was proved by Cilleruelo and Jim\'{e}nez-Urroz \cite{cu} using an idea of Alon.  For $g = 1$, the best bounds 
can be found in \cite{green} and \cite{cill}.  In the case that $h = 2$ and $g \geq 2$, Yu \cite{yu} was able to make some improvements to 
the results of Green \cite{green}.         
In this note we improve (\ref{bound 1}) and make a small improvement upon (\ref{bound 2}).      

\begin{theorem}\label{b3 bound}
(i) Let $g \geq 2$ and $h \geq 4$ be integers.  
If $A \subseteq [N]$ is a $B_h[g]$-set, then 
\[
|A| \leq (1 + o_N (1)) \left( \frac{ x_h h! h g N }{ \pi } \right)^{1/h}
\]
where $x_h$ is the unique real number in $(0 , \pi )$ that satisfies $\frac{ \sin x_h }{x_h } = 
\left( \frac{4}{ 3 - \cos ( \pi / h ) } - 1 \right)^h$.  

\noindent
(ii) If $A \subseteq [N]$ is a $B_3[g]$-set, then for large enough $N$, 
\[
|A| < ( 14.3 g N)^{1/3}.
\]
\end{theorem}

Our improvements for small $h$ are contained in the following table.  

\begin{center}
\begin{tabular}{| c | c | c | } \hline
$h$ & upper bound of \cite{crt}, \cite{cu} & new upper bound \\ \hline
3 & $(16 g N)^{1/3} $ & $(14.3 g N)^{1/3}$  \\
4 &  $(76.8 g N)^{1/4}$ & $( 71.49 g N)^{1/4}$ \\ 
5 & $( 445 .577 g N)^{1/5}$ & $( 413.07 g N)^{1/5}$ \\ 
6 & $( 3054. 7 g N)^{1/6}$ & $(2774.16 g N)^{1/6}$ \\ 
7 & $( 23096.19 g N)^{1/7}$ & $(21294.74 g N)^{1/7}$ \\  \hline

\end{tabular}

\medskip

Table 1: Upper bounds on $B_h[g]$-sets in $\{1,2, \dots , N \}$ for sufficiently large $N$. 
\end{center}

By looking at Table 1, it is clear that Theorem \ref{b3 bound} improves (\ref{bound 1}) for $3 \leq h \leq 6$.  The 
inequality 
\[
\frac{ \sin ( \pi \sqrt{3/h} ) }{ \pi \sqrt{3 / h } } < \left( \frac{4}{3 - \cos ( \pi / h ) } - 1 \right)^h
\]
holds for all $h \geq 3$; a fact that can be verified using Taylor series.  Since $\frac{ \sin x}{x}$ is decreasing on $[0, \pi ]$, we must have $x_h < \pi \sqrt{ 3/ h }$ for all $h \geq 3$ which shows that Theorem \ref{b3 bound} improves (\ref{bound 2}).  
The improvement, however, is $(1 - o_h (1))$ since $\frac{x_h \sqrt{h} }{ \pi \sqrt{3} } \rightarrow 1$ as $h \rightarrow \infty$.    

In the next section we prove Theorem \ref{b3 bound}.  Our arguments rely heavily on \cite{crt} and \cite{green}.   
In Section 3 we introduce an optimization problem that is motivated by our work in Section 2.  


\section{Proof of Theorem 1.1}

First we show how to improve (\ref{bound 1}) using the arguments of \cite{crt} and \cite{green}.
Let $A \subseteq [N]$ be a $B_h[g]$-set where $h \geq 2$.  Define 
$f(t) = \sum_{ a \in A} e^{ i a t }$, $t_h = \frac{2 \pi }{ h N}$, and 
\[
r_h  ( n) = | \{ (a_1 , \dots , a_h ) \in A^h : a_1 + \dots + a_h = n \} |.
\]
The first lemma is a variation of inequality (40) from \cite{green}. 

\begin{lemma}[Green \cite{green}]\label{green lemma 1}
For any $j \in \{1,2, \dots , hN - 1 \}$,
\[
| f ( t_h  j ) | \leq (1 + o_N (1) ) |A| \left( \frac{ \sin ( \pi Q_h ) }{ \pi Q_h } \right)^{1/h}
\]
where $Q_h = \frac{ |A|^h}{h! h g N}$.
\end{lemma}
\begin{proof}
Let $j \in \{1,2, \dots , hN - 1 \}$.  Define $g: \mathbb{Z}_{hN} \rightarrow \{0,1, \dots \}$ by $g(n) = h! g - r_h (n)$.
Following \cite{crt}, we observe that 
\begin{equation}\label{lemma 1 eq} 
f( t_h j)^h = \sum_{n = 1}^{hN} r_h (n) e^{ \frac{2 \pi i n j } {hN}} = - \sum_{n = 1}^{hN} ( h! g - r_h (n) ) 
e^{ \frac{2 \pi i n j }{h N } }.
\end{equation}
Let $\hat{g}$ be the Fourier transform of $g$ so 
$\hat{g} (j) = \sum_{n = 1}^{hN} g(n) e^{  \frac{2 \pi i n j }{hN} }$
for $j \in \mathbb{Z}_{hN}$.  From (\ref{lemma 1 eq}) and the definition of $g$, 
\begin{equation}\label{lemma 1 eq 3} 
| f (t_h j ) |^h = | \hat{g} (j) |.
\end{equation}
Since $A$ is a $B_h[g]$-set, the inequality $0 \leq g(n) \leq h! g$ holds for all $n$.  Furthermore, 
$\sum_{n =1}^{hN} g(n ) = h! g h N  - |A|^h$.  Lemma 26 of \cite{green} gives 
\begin{equation}\label{lemma 1 eq 2}
| \hat{g} (j) | \leq h ! g \left| \frac{ \sin ( \frac{ \pi }{hN} ( \frac{ h! hg N - |A|^h }{h! g } + 1 ) ) }{ \sin ( \frac{ \pi }{ hN} ) } \right|
= h! g \left| \frac{ \sin ( \pi Q_h - \frac{ \pi }{hN} ) }{ \sin ( \frac{ \pi }{hN } ) } \right|.
\end{equation}
By (\ref{bound 1}), the value $Q_h$ satisfies $0 \leq Q_h \leq 1$ for all $N$.  Therefore, 
\[
| \hat{g} (j) | \leq h! g (1 + o_N (1)) \frac{ \sin ( \pi Q_h )}{ \pi / hN } = (1 + o_N(1)) |A|^h \frac{ \sin ( \pi Q_h ) }{ \pi Q_h }.
\]
Combining this inequality with (\ref{lemma 1 eq 3}), we get
\[
| f (t_h j) | \leq (1 + o_N (1)) |A| \left( \frac{ \sin ( \pi Q_h ) }{ \pi Q_h } \right)^{1/h}
\]   
which completes the proof of the lemma.  
\end{proof}

\bigskip

Again following \cite{crt}, we need to choose a function $F(x) = \sum_{j = 1}^{ h N  } b_j \cos ( j x)$ such that 
\[
\sum_{a \in A} F \left( \left( a - \frac{N+1}{2} \right) t_h \right)
\]
is large and $\sum_{j = 1}^{ hN } |b_j|$ is small.  For $h \geq 3$, the function 
$F(x)= \frac{1}{ \cos ( \pi / h ) } \cos x$ gives
\[
\sum_{a \in A} F \left( \left( a - \frac{N+1}{2} \right) t_h \right) \geq |A|
\]
and $\sum_{j = 1}^{h N } |b_j|  = \frac{1}{ \cos (\pi / h ) }$.  This is the function that is used in \cite{crt}.  
We will choose a different function $G$ that does better than $F$ and still has a simple form.  
Let 
\begin{equation}\label{def of G}
G(x) = \left(  \frac{ 2}{ 3 - \cos ( \pi / h ) } \right) \frac{1}{ \cos ( \pi / h ) } \cos (x) 
- \left( 1 - \frac{2}{3 - \cos ( \pi / h ) } \right) \frac{1}{ \cos ( \pi / h ) } \cos ( h x).
\end{equation}
The minimum value of $G(x)$ on the interval $[ - \frac{\pi }{ h}  , \frac{ \pi }{ h} ]$ is $\frac{1}{ \cos ( \pi / h ) } \left( \frac{ 4}{ 3 - \cos ( \pi / h )} - 1 \right)$ and so  
\begin{equation}\label{sdf}
\sum_{a \in A} G  \left( \left( a - \frac{N+1}{2} \right) t_h \right) \geq \frac{1}{ \cos ( \pi / h ) } \left( \frac{ 4}{ 3 - \cos ( \pi / h )} - 1 \right) |A|.
\end{equation}
Here we are using the fact that $| ( a - (N+1)/2 ) t_h | < \frac{ \pi }{h}$ for any $a \in A$.  
If the constants $c_j$ are defined by $G(x) = \sum_{j =1 }^{hN} c_j \cos (jx)$, 
then $\sum_{j = 1}^{h N } |c_j|  = \frac{1}{ \cos ( \pi / h ) }$.  
Using (\ref{sdf}), we have  
\begin{eqnarray*}
\frac{1}{ \cos ( \pi / h ) } \left( \frac{ 4}{ 3 - \cos ( \pi / h )} - 1 \right) |A| & \leq & \sum_{a \in A} G  \left( \left( a - \frac{N+1}{2} \right) t_h \right) \\
& = & \textup{Re} \left( \sum_{j=1}^{ hN } c_j \sum_{a \in A} e^{ ( a - (N+1)/2) \frac{2 \pi i j } {hN }} \right) \\
& \leq & \sum_{j = 1}^{h N } |c_j| | f ( t_h j )| \\
& \leq & \frac{1}{ \cos ( \pi / h ) } (1 + o_N (1) ) |A| \left( \frac{ \sin ( \pi Q_h ) }{ \pi Q_h } \right)^{1/h}
\end{eqnarray*}
where in the last line we have used Lemma \ref{green lemma 1} and $\sum_{j =1}^{ hN } |c_j| = \frac{1}{ \cos ( \pi / h ) }$.  
Some rearranging gives 
\begin{equation}\label{new bound 1}
\left( \frac{4}{ 3 - \cos ( \pi / h ) } - 1 \right)^h \leq (1 + o_N (1)) \frac{ \sin ( \pi Q_h ) }{ \pi Q_h }.
\end{equation}
We remark that $\frac{4}{ 3 - \cos ( \pi /h ) } - 1 > \cos ( \pi / h )$ is equivalent to $(1 - \cos ( \pi / h ) )^2 > 0$.  
The point of this is that using $G$ defined by (\ref{def of G}) instead of $F(x) = \frac{1}{ \cos ( \pi / h ) } \cos x$ (which would give the value 1 on the left hand side of (\ref{new bound 1})) does lead to a better upper bound.   
  
Recalling that $0 \leq Q_h \leq 1$, lower bounds on $\frac{ \sin ( \pi Q_h )}{ \pi Q_h}$ translate to upper bounds 
on $\pi Q_h$.  Let $x_h$ be the unique real number in the interval $( 0 , \pi )$ that satisfies 
\[
\left( \frac{4}{ 3 - \cos ( \pi / h ) } - 1 \right)^h = \frac{ \sin (x_h) }{x_h }.
\]
Then by (\ref{new bound 1}), $\pi Q_h \leq (1 + o_N(1)) x_h$ since the function $\frac{ \sin x}{x}$ is decreasing on $[0, \pi ]$.  
We can rewrite $\pi Q_h \leq (1 + o_N(1)) x_h$ as 
\begin{equation}\label{new bound 2}
|A| \leq (1 + o_N (1)) \left( \frac{x_h h! h g N }{ \pi } \right)^{1/h}.
\end{equation}
The upper bounds obtained from (\ref{new bound 2}) for $h \in \{4,5,6,7 \}$ are given in Table 1.  We have chosen to round the values so that all of the bounds in Table 1 hold for sufficiently large $N$.   
In particular, (\ref{new bound 2}) implies that a $B_3[g]$-set $A \subseteq  [N]$ has at most 
$(14.65 g N)^{1/3}$ elements.  We can improve this bound by considering the distribution of $A$ in the 
interval $[ N]$.    

Assume now that $A$ is a $B_3 [g]$-set.  
Let $\delta$ be a real number with $0 < \delta < \frac{1}{4}$ and set $l = \lfloor \frac{1}{2 \delta } \rfloor$.  For 
$1 \leq k \leq l$, let 
\[
C_k = \left( A \cap ( ( k -1 ) \delta N , k \delta N ] \right) \cup \left( A \cap [ ( 1- k \delta ) N , ( 1 - ( k - 1) \delta ) N ) \right).
\]
The definition of $l$ ensures that the sets $C_1 , \dots , C_l$ together with $A \cap ( l \delta N  , (1 - l \delta ) N )$ form a partition of $A$.  
Using the same counting argument that is used to obtain (\ref{trivial bound}), we show that if 
some $C_k$ contains a large proportion of $A$, then $|A| \leq ( 14.295 g N)^{1/3}$.  To this end, 
define real numbers $\alpha_1 ( \delta ) , \dots , \alpha_l ( \delta )$ by 
\begin{equation}\label{alphas}
\alpha_k ( \delta ) |A| = |C_k|
\end{equation}
for $1 \leq k \leq l$.  The value $\alpha_k ( \delta )$ represents the proportion of $A$ that is contained in the union
$( ( k - 1) \delta N , k\delta N] \cup[ ( 1 - k \delta )N  , (1 - ( k - 1) \delta ) N )$.  

\begin{lemma}\label{ck lemma}
If $0 < \delta < \frac{1}{4}$, $l = \lfloor \frac{1}{2 \delta } \rfloor$, and $\alpha_1 ( \delta ) , \dots , \alpha_l ( \delta)$ are 
defined by (\ref{alphas}), then for any $N > \frac{2 }{ \delta }$ and $1 \leq k \leq l$, 
\[
|A| \leq \left( \frac{ 72 g \delta N}{ \alpha_k( \delta )^3 } \right)^{1/3}.
\]
\end{lemma}
\begin{proof}
Let $1 \leq k \leq l$ and consider $C_k$.  Since $C_k$ is a $B_3[g]$-set,  
\begin{equation}\label{asdf}
\binom{ |C_k| + 3 - 1 }{3 } \leq g |C_k + C_k + C_k|
\end{equation}
where $C_k + C_k + C_k = \{ a + b + c : a,b,c \in C_k \}$.  The set $|C_k + C_k + C_k |$ is contained in the union of the intervals
\begin{center}
$[ 3(k - 1 ) \delta N , 3 k \delta N]$, $[ (1 + ( k - 2) \delta ) N , (1 +  ( k + 1) \delta )N]$, 

\medskip

$[ (2 - ( k + 1) \delta )N , (2 - ( k - 2) \delta )N]$, and $[(3 - 3k \delta )N , (3 - 3(k - 1) \delta )N ]$.
\end{center}
Each of these four intervals has length $3 \delta N$ so 
$| C_k  + C_k + C_k| \leq 12 \delta N$.  Combining this inequality with (\ref{asdf}) we have 
$\binom{ |C_k | + 2 }{3} \leq 12 g \delta N$
which implies $\alpha_k ( \delta ) |A| = |C_k| \leq ( 3! 12 g \delta N )^{1/3}$.    
\end{proof}

\smallskip

Now we consider two cases.

\smallskip

\noindent
\textbf{Case 1:} For some $0 < \delta < \frac{1}{4}$ and $1 \leq k \leq l = \lfloor \frac{1}{2 \delta } \rfloor$, we have 
\[
\left( \frac{ 72 \delta }{14.295} \right)^{1/3} < \alpha_k ( \delta ).
\]

\medskip

In this case, we apply Lemma \ref{ck lemma} to get $|A| \leq ( 14.295 g N )^{1/3}$ and we are done.

\bigskip

\noindent
\textbf{Case 2:} For all $0 < \delta < \frac{1}{4}$ and $1 \leq k \leq l = \lfloor \frac{1}{2 \delta } \rfloor$, we have 
\begin{equation}\label{case 2}
  \alpha_k ( \delta ) \leq \left( \frac{ 72 \delta }{14.295} \right)^{1/3}.
\end{equation}

\medskip

Let $H(x) = 1.6 \cos x - 0.3 \cos 3x + 0.1 \cos 6x$.  Partition the interval $[ - \pi / 3 , \pi / 3]$ into 128 subintervals 
$I_1 , \dots , I_{128}$ of equal width so 
\[
I_j = \left[ - \frac{ \pi }{ 3 } +   \frac{ 2 \pi (j-1)}{3 \cdot 128 } , - \frac{  \pi }{ 3 } +  \frac{2 \pi j }{3 \cdot 128} \right]
\]
for $1 \leq j \leq 128$.  Let $v_j = \min_{x \in I_j} H(x)$ for $1 \leq j \leq 128$.  Since $H$ is an even function, 
$v_j = v_{128 - j + 1}$ for $1 \leq j \leq 64$.  The values $v_j$ can be approximated numerically. They satisfy 
\begin{equation}\label{vs}
v_1 < v_2 < v_3 < v_4  < v_5 < v_{35} \leq v_j
\end{equation}
for all $ 6 \leq j \leq 64$.  The sum 
\begin{equation}\label{sum H}
\sum_{a \in A} H \left( \left( a - \frac{N+1}{2} \right) t_3 \right)
\end{equation}
is minimized when $J = \left\{    \left( a - \frac{N+1}{2} \right) t_3 : a \in A \right\}$
contains as many elements as possible in $I_1 \cup I_2 \cup \dots \cup I_5$ and the remaining 
elements of $J$ are contained in $I_{35}$.  This follows from (\ref{vs}).  Furthermore, in order to minimize 
(\ref{sum H}), $J$ must intersect $I_1$ in as many elements as possible, and the remaining elements in $J$ intersect 
$I_2$ in as many elements as possible, and so on.  
By (\ref{case 2}) with $\delta = 1/128$, 
\[
\alpha_k (  1/128 ) \leq \left( \frac{ 72 (1/128) }{14.295} \right)^{1/3} 
\]
thus, 
\[
|J \cap I_1 | \leq \left( \frac{ 72 (1/128) }{14.295} \right)^{1/3} |A|.
\]
Similarly, by (\ref{case 2}) with $\delta  = j/128$ for $j \in \{2,3,4,5 \}$, 
\[
  \alpha_k (  j/128 ) \leq \left( \frac{ 72 (j/128) }{14.295} \right)^{1/3}.
\]
We conclude that 
\[
|J \cap ( I_1 \cup I_2 \cup \dots \cup I_j ) | \leq \left( \frac{ 72 (j/128) }{14.295} \right)^{1/3} |A|
\]  
for $1 \leq j \leq 5$.  From this inequality and (\ref{vs}), we deduce that 
\begin{eqnarray*}
\sum_{a \in A} H \left( \left( a - \frac{N+1}{2} \right) t_3 \right) & \geq & \sum_{j = 1}^{5} v_j 
\left( 
\left( \frac{ 72 (j/128) }{14.295} \right)^{1/3}  - 
\left( \frac{ 72 ((j-1)/128) }{14.295} \right)^{1/3} 
\right) |A| \\ 
& + & v_{35} \left( 1 - \left( \frac{ 72 (5/128) }{14.295} \right)^{1/3} \right) |A| > 1.2455 |A|.  
\end{eqnarray*}

Using 1.2455 in the derivation of (\ref{new bound 1}) instead of 
$\frac{1}{ \cos ( \pi / 3) } \left( \frac{4}{3 - \cos ( \pi  / 3 ) } - 1 \right)$ gives 
\[
1.2455 |A| \leq \frac{1}{ \cos ( \pi / 3 ) } (1 + o_N (1)) |A| \left( \frac{ \sin ( \pi Q_3 ) }{ \pi Q_3 } \right)^{1/3} .
\]
This inequality can be rewritten as 
\[
\left( \frac{1.2455}{2} \right)^3 \leq (1 + o_N (1)) \left( \frac{ \sin ( \pi Q_3 ) }{ \pi Q_3 } \right)
\]
which leads to the bound $|A| < (14.296 g N)^{1/3}$ for large enough $N$.  


\section{An optimization problem}

In this section we introduce an optimization problem that is motivated by (\ref{sdf}) from the previous section.  

Given integers $K$ and $h \geq 2$, define 
\[
\mathcal{F}_{K,  h} = \left\{ \sum_{j=1}^{ K } b_j \cos ( jx ) : \sum_{j = 1}^{ K } | b_j | = \frac{1}{ \cos  ( \pi / h )} \right\}.
\]
For $A \subseteq [N]$ and $F \in \mathcal{F}_{K , h }$, define 
\[
w_F (A) = \sum_{a \in A} F \left( \left(  a - \frac{N+1}{2} \right) \frac{2 \pi }{ h  N} \right)
\]
and
\[
\psi ( N , K , h) = \min_{  A \subseteq [N] , A \neq \emptyset} \sup \left\{ \frac{ w_F(A) }{|A|} : F \in \mathcal{F}_{ K , h } \right\}.
\]
Our interest in $\psi ( N , K ,h )$ is due to the following proposition.  

\begin{proposition}\label{tech lemma}
If $A \subseteq [N]$ is a $B_h[g]$-set and $K \leq hN$, then 
\[
|A| \leq  (1 + o_N (1)) \left( \frac{ y_h h! h g N}{ \pi } \right)^{1/h}
\]
where $y_h$ is the unique real number in $[0, \pi ]$ with 
$\frac{ \sin y_h }{y_h } = \left( \cos( \pi / h ) \psi ( N , K , h ) \right)^h$.   
\end{proposition}

The function $G$ defined by (\ref{def of G})  
shows that 
\begin{equation*}
\psi (N , h , h) \geq \frac{1}{ \cos ( \pi / h ) } \left( \frac{4}{3 - \cos ( \pi / h ) } - 1 \right).
\end{equation*}

When $h =3$, this gives $\psi (N , 3 , 3) \geq 1.2$ which implies $\psi (N , 6 , 3 ) \geq 1.2$.  This is because 
the collection of functions $\mathcal{F}_{3,3}$ is a subset of $ \mathcal{F}_{6,3}$.  
By considering more than one function, we can improve the bound $\psi (N , 6 , 3) \geq 1.2$.  The method by which
we achieve this can be stated just as easily for general $K$ and $h$ so we do so.         

To estimate $\psi ( N , K , h )$, we will consider finite subsets of $\mathcal{F}_{K , h }$.  Given a subset 
$\mathcal{F}_{K,h} ' \subseteq \mathcal{F}_{k , h }$, we obviously have 
\begin{equation}\label{f f'}
\sup \left\{ \frac{ w_F (A) }{ |A| } : F \in \mathcal{F}_{K,h}' \right\} 
\leq   
\sup \left\{ \frac{ w_F (A) }{ |A| } : F \in \mathcal{F}_{K,h} \right\}
\end{equation}
for every $A \subseteq [N]$ with $A \neq \emptyset$.  When $\mathcal{F}_{K,h}'$ is finite, then the 
supremum on the left hand side of (\ref{f f'}) can be replaced with the minimum.  
Let $m$ be a positive integer and partition the interval $[ - \pi / h , \pi / h ]$ into $m$ subintervals $I_1^m , \dots , I_m^m$ where 
\[
I_j^m = \left[ - \frac{ \pi }{h} + \frac{ 2 \pi ( j - 1) }{hm} , - \frac{ \pi }{h} + \frac{ 2 \pi j  }{h m } \right]
\]
for $1 \leq j \leq m$.  Any $F \in \mathcal{F}_{K,h}$ is continuous and thus obtains its minimum value on $I_j^m$.  
Given $F \in \mathcal{F}_{K,h}$, define 
\[
v_{m,j} (F) = \min_{x \in I_j^m } F(x).
\]
Given $A \subseteq [N]$, define
\[
\alpha_{m,j} (A) = \frac{1}{ |A| } \left| \left\{ ( a - \frac{N+1}{2} ) \frac{2 \pi }{ hN } : a \in    A \right\} \cap I_j^m \right|.
\]
With this notation, we have that for any $A \subseteq [N]$ and $F \in \mathcal{F}_{K,h}$, 
\[
w_F(A) \geq \sum_{j = 1}^{m} \alpha_{m , j}(A) |A| v_{m,j} (F ).
\]
Therefore, given a finite set $\{ F_1 , \dots , F_n \} \subseteq \mathcal{F}_{K,h}$, 
\[
\psi ( N , K , h ) \geq \min_{ A \subseteq [N] , A \neq \emptyset} \max 
\left\{ \sum_{j = 1}^{m} \alpha_{m,j} (A) v_{m,j} (F_k) : 1 \leq k \leq n \right\}.
\]

We now put the above discussion to use by proving the following result.  

\begin{theorem}\label{psi 3 bound}
For sufficiently large $N$, the function $\psi (N , 6 , 3)$ satisfies the estimate 
\[
\psi (N , 6 , 3) \geq 1.2228.
\]
\end{theorem}
\begin{proof}
Let 

\medskip

$F_1 (x) = 1.7 \cos x - 0.3 \cos 3x$, $F_2 (x) = 1.6 \cos x - 0.3 \cos 3x + 0.1 \cos 6x$,

\medskip

$F_3 (x) = 1.5 \cos x - 0.4 \cos 3x + 0.1 \cos 6x$, $F_4 (x) = 1.2 \cos x - 0.6 \cos 3x + 0.2 \cos 6x$, 

\medskip

$F_5 (x) = -2 \cos 3x$,

\medskip

\noindent
and $\mathcal{F} = \{ F_1 , F_2 , F_3, F_4 , F_5 \}$.  Observe that $\mathcal{F} \subseteq \mathcal{F}_{6,3 }$.  
We take $m = 12$ and we must compute the numbers 
$v_{12,j} (F_k)$ for $1 \leq j \leq 12$ and $1 \leq k \leq 5$.  Since each $F_k$ is an even function, 
$v_{12,j} (F_k) = v_{12 , 12 - j + 1} (F_k)$ for $1 \leq j \leq 6$.  
To prove Theorem \ref{psi 3 bound}, we will only need to estimate these values from below.  

Let $A \subseteq [N]$ with $A \neq \emptyset$.  We assume that no element of the form $( a - \frac{N+1}{2} ) \frac{2 \pi }{3N }$ 
is contained in two of the intervals $I_1^{12} , \dots , I_{12}^{12}$.  For large $A$, this will not affect $|A|$, at least in an asymptotic sense. 
Under this assumption, the non-negative real numbers 
$\alpha_{12,1} (A) , \dots , \alpha_{12,12} (A)$ satisfy 
\[
\alpha_{12 , 1}(A) + \dots + \alpha_{12, 12} (A) = 1.
\]
We will consider several cases which depend on the distribution of $A$.
For notational convenience, we write $\alpha_j$ for $\alpha_{12 , j } (A)$.     

\bigskip

\noindent
\textbf{Case 1:} $\alpha_1 + \alpha_{12} \leq 0.6$.

\smallskip

Here we will use the function $F_1(x)$.    
Lower estimates on the $v_{ 12 , j} (F_1)$ are 
\begin{center}
$v_{12 , 1} (F_1) \geq 1.15$, ~$v_{12,2} (F_1) \geq 1.3525$, ~$v_{12,3} (F_1) \geq 1.4522$, 

\medskip

$v_{12,4} (F_1) \geq 1.4474$, ~$v_{12, 5} (F_1) \geq 1.4143$,~ and ~$v_{12,6} (F_1) \geq 1.4$.
\end{center}
In fact, these values satisfy
\[
v_{12,1}(F_1) \leq v_{12,2} (F_1) \leq v_{12,6} (F_1) \leq v_{12,5} (F_1) \leq v_{12,4} (F_1) \leq v_{12,3} (F_1).
\]
Since $\alpha_1 + \alpha_{12} \leq 0.6$, we must have 
\[
w_{F_1} (A) \geq ( 0.6 v_{12,1} (F_1) + 0.4 v_{12,2} (F_1) )|A| \geq ( 0.6(1.15) + 0.4 ( 1.3525) )|A| > 1.23 |A|.
\]

\noindent
\textbf{Case 2:} $0.6 \leq \alpha_1 + \alpha_{12} \leq 0.7$.

\smallskip

Here we use the function $F_2 (x)$.   A close look at Case 1 shows that if $v_{12,1} (F_2)$ is one of the two 
smallest values in the set $\{ v_{12,j} (F_2) : 1 \leq j \leq 6 \}$, then essentially the same 
estimate applies.  The two smallest values 
are $v_{12,1}(F_2) \geq 1.2$ and $v_{12,4} (F_2) \geq 1.2834$.  
Since $0.6 \leq \alpha_1 + \alpha_{12} \leq 0.7$, 
\[
w_{F_2} (A) \geq ( 0.7 ( 1.2) + 0.3( 1.2834) )|A| > 1.225 |A|.
\]

\noindent
\textbf{Case 3:} $0.7 \leq \alpha_1 + \alpha_{12} \leq 0.8$.

\smallskip

Here we use the function $F_3 (x)$.   In this range of $\alpha_1 + \alpha_{12}$, 
our estimate behaves a bit differently.  Lower estimates on the $v_{12,j} (F_3)$ are  
\begin{center}
$v_{12 , 1} (F_3) \geq 1.25$, ~ $v_{12,2} (F_3) \geq 1.299$, ~ $v_{12,3} (F_3) \geq 1.199$, 

\medskip

$v_{12,4} (F_3) \geq 1.1595$, ~$v_{12, 5} (F_3) \geq 1.1595$, ~and ~$v_{12,6} (F_3) \geq 1.18$.
\end{center}
In this case, $w_{F_3} (A)$ will be minimized when $\alpha_1 + \alpha_{12}$ is as small as possible.  In the previous 
two cases, $w_{F_i} (A)$ was minimized when $\alpha_1 + \alpha_{12}$ was as large as possible.  
We conclude that
\[
w_{F_3} (A) \geq ( 0.7 ( 1.25) + 0.3( 1.1595) )|A| > 1.2228 |A| .
\]

\noindent
\textbf{Case 4:} $0.8 \leq \alpha_1 + \alpha_{12} \leq 0.9$.

\smallskip

In this case we use the function $F_4(x)$.  Lower estimates on the $v_{12,j} (F_4)$ are  
\begin{center}
$v_{12 , 1} (F_4) \geq 1.3909$, ~ $v_{12,2} (F_4) \geq 1.1192$, ~ $v_{12,3} (F_4) \geq 0.8392$, 

\medskip

$v_{12,4} (F_4) \geq 0.7276$, ~$v_{12, 5} (F_4) \geq 0.7264$, ~and ~$v_{12,6} (F_4) \geq 0.7621$.
\end{center}
We have 
\[
w_{F_4} (A) \geq ( 0.8 ( 1.3909) + 0.2( 0.7264) )|A| > 1.25 |A| .
\]

\noindent
\textbf{Case 5:} $0.9 \leq \alpha_1 + \alpha_{12} \leq 1$.

\smallskip

Lower estimates on the $v_{12,j} (F_5)$ are
\begin{center}
$v_{12 , 1} (F_5) \geq 1.73$, ~ $v_{12,2} (F_5) \geq 1$, ~ $v_{12,3} (F_5) \geq -.01$, 

\medskip

$v_{12,4} (F_5) \geq -1$,~ $v_{12, 5} (F_5) \geq -1.8$,~ and ~ $v_{12,6} (F_5) \geq -2$.
\end{center}
As in Cases 3 and 4, $w_{F_5} (A)$ is minimized when $\alpha_1 + \alpha_{12}$ is as small as possible.  
Hence,  
\[
w_{F_5} (A) \geq ( 0.9 ( 1.73) + 0.1( -2) )|A| > 1.35 |A|.
\]

\medskip

In all five cases, we can find a function $F_i \in \mathcal{F}$ such that $w_{F_i} (A) > 1.2228 |A|$.  This completes the proof of 
Theorem \ref{psi 3 bound}.    
\end{proof}


\section{Concluding Remarks}

Although it is an improvement of $\psi (N , 6 , 3) \geq 1.2$, 
Theorem \ref{psi 3 bound} is not enough to prove part (ii) of Theorem \ref{b3 bound}.  
The improvement on $B_3[g]$-sets uses the $B_3[g]$ property to increase the 1.2 to 1.2455
which exceeds the 1.2228 provided by Theorem \ref{psi 3 bound}.  Similar arguments can be done for $B_h[g]$-sets with $h >3$, but the improvements in the results of Table 1 are minimal.  
Aside from $B_3[g]$-sets, the bounds in Table 1 come from lower bounds on $\psi(N ,h , h)$ together with 
Lemma \ref{green lemma 1}.    

The function $\psi ( N , K , h )$ is relevant to an inequality of Cilleruelo.   
Let $A$ be a finite set of positive integers.  For an integer $h \geq 2$, let 
\begin{center}
$r_h (n) = | \{ ( a_1 , \dots , a_h) \in A^h : a_1 + \dots + a_h = n \}|$ and $R_h ( m) = \displaystyle\sum_{ n=1}^{m} r_h (m)$.
\end{center}
Generalizing the argument of \cite{crt}, Cilleruelo proved the following result.

\begin{theorem}[Cilleruelo \cite{cill}]\label{cill thm}
Let $A \subseteq [N]$, $h \geq 2$ be an integer, and $\mu$ be any real number.  For any positive integer $H = o(N)$, 
\[
\sum_{n = h }^{hN + H} | R_h (n) - R_h (n - H) - \mu | \geq ( L_h + o(1) ) H |A|^h
\]
where $L_2 = \frac{4}{ ( \pi + 2)^2 }$ and $L_h = \cos^h ( \pi / h )$ for $h > 2$.
\end{theorem}

By slightly modifying the argument in \cite{cill} that is used to prove Theorem \ref{cill thm}, it is easy to prove the next proposition.

\begin{proposition}
Let $A \subseteq [N]$, $h \geq 2$ be an integer, and $\mu$ be a real number.  For any positive integers 
$H = o(N)$ and $K \leq \frac{N}{H}$, 
\[
\sum_{n = h }^{hN + H} | R_h (n) - R_h (n - H) - \mu | \geq ( \psi (N , K , h )^h  L_h + o(1) ) H |A|^h
\]
where $L_2 = \frac{4}{ ( \pi + 2)^2 }$ and $L_h = \cos^h ( \pi / h )$ for $h > 2$.
\end{proposition}

For instance, Theorem \ref{psi 3 bound} gives
\[
\sum_{n = 3}^{3N + H} | R_3 (n)  - R_3 (n - H ) - \mu | \geq 
(1.2228^3 L_3 + o(1) ) H |A|^3.
\]


\section{Acknowledgment}

The author would like to thank Mike Tait for helpful discussions.  


\end{document}